\documentclass[11pt]{amsart}
\usepackage{amssymb, latexsym, amsmath, amsfonts, hyperref, enumitem, tikz, fancyhdr,bm}
\usepackage[all]{xy}
\usepackage[mathscr]{euscript}
\usepackage{blindtext}
\usepackage{mathtools}
\usepackage[english]{babel}
\DeclarePairedDelimiterX{\norm}[1]{\lVert}{\rVert}{#1}
\newtheorem{theorem}{Theorem}[section]

\newtheorem{lemma}[theorem]{Lemma}
\newtheorem{definition}[theorem]{Definition}

\newtheorem{corollary}[theorem]{Corollary}

\usepackage{tikz}
\usetikzlibrary{calc}

\hypersetup{
  colorlinks,
  citecolor=black,
  linkcolor=black,
  urlcolor=blue}
\theoremstyle{plain}

  {\end{enumerate}}

\makeatletter
\newcommand*\rel@kern[1]{\kern#1\dimexpr\macc@kerna}
\newcommand*\widebar[1]{%
  \begingroup
  \def\mathaccent##1##2{%
    \rel@kern{0.8}%
    \overline{\rel@kern{-0.8}\macc@nucleus\rel@kern{0.2}}%
    \rel@kern{-0.2}%
  }%
  \macc@depth\@ne
  \let\math@bgroup\@empty \let\math@egroup\macc@set@skewchar
  \mathsurround\z@ \frozen@everymath{\mathgroup\macc@group\relax}%
  \macc@set@skewchar\relax
  \let\mathaccentV\macc@nested@a
  \macc@nested@a\relax111{#1}%
  \endgroup
}
\makeatother

\hypersetup{
  colorlinks,
  citecolor=blue,
  linkcolor=blue,
  urlcolor=blue}

\begin{document}
\title{On Joint Functional Calculus For Ritt Operators}
%{\let\thefootnote\relax\footnote{}}
%{\let\thefootnote\relax\footnote{The results in this paper are from the author's thesis ``The Carath\'{e}odory-Fej\'{e}r Interpolation Problems and the von Neumann Inequality" submitted to the Indian Institute of Science, Bangalore-560012.}}

%\thanks{The first named author is supported by Science and Engineering Research Board, DST, Government of India.}
\thanks{The second author is supported by Council for Scientific and Industrial Research, MHRD, Government of India.}
\author{parasar Mohanty}
\author{Samya Kumar Ray}
\address{Parasar Mohanty: Department of Mathematics and Statistics, Indian Institute of Technology, Kanpur-208016}
\email{parasar@math.iitk.ac.in}

\address{Samya Kumar Ray: Department of Mathematics and Statistics, Indian Institute of Technology, Kanpur-208016}
\email{samya@math.iitk.ac.in}

\pagestyle{headings}

\begin{abstract}
In this paper, we study joint functional calculus for commuting $n$-tuple of Ritt operators. We provide an equivalent characterisation of boundedness for joint functional calculus for Ritt operators on $L^p$-spaces, $1< p<\infty,$ where each of them admits a bounded functional calculus. We also investigate joint similarity problem for commuting $n$-tuple of Ritt operators. We get our results by proving a suitable multivariable transfer principle between sectorial and Ritt operators as well as an appropriate joint dilation result in a  general setting.
\end{abstract}
\maketitle
\section{Introduction} 
%\label{}
Let $X$ be a Banach space and $T$ be a linear operator on $X.$ In many instances it is important to assign a meaning to the symbol $f(T)$ for a bounded holomorphic function $f$ defined on an appropriate domain.  From celebrated von Neumann inequality \cite{vN},  we can associate bounded functional calculus  for every contraction when $X$ is a Hilbert space. More precisely, if  $f$ is any holomorphic function with poles outside closure of the unit disc $\overline{\mathbb D}$  and $T$ a contraction,  then 
\begin{equation*}
\|f(T)\|\leq\|f\|_{\infty,\mathbb{D}}. 
\label{VN}
\end{equation*}

Above inequality follows for a   commuting pair of contractions $(T_1,T_2)$ on a Hilbert space from a theorem by  Ando \cite{AT}. However,  Varopoulos  \cite{VA} provided an example of the failure of this inequality for commuting $3$-tuple of contractions. 

McIntosh and his coauthors \cite{CO}, \cite{MC} developed a notion of $H^\infty$-functional calculus for sectorial operators on Banach spaces. Sectorial operators have been studied extensively. It has many applications in partial differential  equations and harmonic analysis \cite{WE1}. The notion of joint functional calculus for sectorial operators with commuting resolvents is due to \cite{AL}, which again has important applications in studying maximal regularity problems \cite{KL}, \cite{LLL}.

In \cite{LeM2} Christian Le Merdy  initiated the study of functional calculus for Ritt operators. He and  his coauthors in \cite{AFL}, \cite{ALM} and \cite{LeM2},  have obtained various results in the lines of sectorial operators and also  characterised  Ritt operators on $L^p$-spaces having bounded $H^\infty$-functional calculus. 

In this paper, we develop a notion of joint functional calculus for commuting Ritt operators and obtain equivalent criterion for commuting tuple of Ritt operators to admit joint bounded functional calculus on $L^p$-space, $1<p<\infty,$ where each of them also satisfies a bounded functional calculus.
We apply our methods to study joint similarity problem for commuting tuple of Ritt operators acting on a Hilbert space.

\section{Notations and main results}\label{Bi}
Let $(\Omega,\mathbb{F},\mu)$ be a $\sigma$-finite measure space and $1\leq p\leq\infty$. We denote $L^p(\Omega,X)$ to be the usual Bochner space. Given any open set $D\subseteq\mathbb{C}^n,$ we define $H^\infty(D)$ to be the set of all bounded holomorphic maps equipped with the supremum norm. For any set $K$ and $f:K\to \mathbb{C}$,  denote $\|f\|_{\infty,K}:=\sup\{|f(z)|:z\in K\}.$ If $a\in\mathbb{C},$ we denote the open disc of radius $r$ with center $a$ to be $D(a,r).$ For any set $S$ we denote $S^n$ by its $n$-times cartesian product and  $d(S_1,S_2)$ be the usual distance between $S_1,S_2\subseteq\mathbb{C}$, i.e. $d(S_1,S_2):=\inf\{|z_1-z_2|:z_1\in S_1,z_2\in S_2\}.$
\subsection{Sectorial operators:} We recall some preliminaries on sectorial operators and associated joint $H^\infty$-functional calculus. We recommend the interested readers \cite{AL}, \cite{DV}, \cite{FM}, \cite{KK}, \cite{KL} and \cite{LLL} for more on this direction.
\begin{definition}[Sectorial operator]\label{SEC}
For any $\omega\in(0,\pi),$ let $\Sigma_{\omega}:=\{z\in\mathbb{C}\setminus\{0\}:|arg\;z|<\omega\}$ be the open sector of an angle $2\omega$ around the positive real axis $(0,\infty)$ (see FIGURE. \ref{fig:S}).
We say that a densely defined closed operator $A:D(A)\subseteq X\rightarrow X$ with domain $D(A)$ is sectorial of type $\omega\in(0,\pi)$ if  we have
\begin{itemize}
\item[(i)] $\sigma(A)\subseteq\overline{\Sigma_\omega}.$ 
\item[(ii)]For any $\nu\in(\omega,\pi),$ the set $\{zR(z,A): z\in\mathbb{C}\setminus\overline{\Sigma_\nu}\}$ is bounded, where $R(.,A)$ is the resolvent operator of $A.$
\end{itemize}

\end{definition}

Let $\nu\in(0,\pi)$ and $\Gamma_{\nu}$ be the boundary of $\Sigma_{\nu},$ oriented counter-clockwise (see FIGURE. \ref{fig:S}).
For $\theta_i\in(0,\pi),$ $1\leq i\leq n,$ denote $H_0^\infty\big(\prod_{i=1}^n\Sigma_{\theta_i}\big)$ to be the set of all bounded holomorphic 
functions $f:\prod_{i=1}^n\Sigma_{\theta_i}\to\mathbb{C}$ with the property that there exists constants $C,s>0,$ depending only on $f,$ such that
\begin{equation*}
|f(z_1,\dots,z_n)|\leq C\prod_{i=1}^n\frac{|z_i|^s}{1+|z_i|^{2s}}\ \ ,\ \text{for all}\ (z_1,\dots,z_n)\in\prod_{i=1}^n\Sigma_{\theta_i}.
\end{equation*}
Let $\textbf{A}:=(A_1,\dots,A_n)$ be an $n$-tuple of operators with mutually commuting resolvents such that each $A_i$ is sectorial of type $\omega_i\in(0,\pi),\ 1\leq i\leq n.$ Let $\omega_i<\theta_i<\pi$ and $\nu_i\in (\omega_i,\theta_i),$
$1\leq i\leq n.$ One defines 
\begin{equation}\label{JO}
f(\textbf{A}):=\Big(\frac{1}{2\pi i}\Big)^n\int_{\prod_{i=1}^n\Gamma_{\nu_i}}f(z_1,\dots,z_n)\prod_{i=1}^nR(z_i,A_i)dz_i.
\end{equation}
The decay on the resolvent operators ensures that the integral in \eqref{JO} is absolutely convergent and by Cauchy's theorem, it is independent of $\nu_i,$ $1\leq i\leq n.$
Let $\Phi_{\textbf{A}}:H_0^\infty(\prod_{i=1}^n\Sigma_{\theta_i})\to B(X)$ be defined as \[\Phi_{\textbf{A}}(f):=f(\textbf{A}).\] Then, the map $\Phi_{\textbf{A}}$ is an algebra homomorphism \cite{DV}.

\begin{figure}
\begin{center}
\begin{tikzpicture}
\draw[gray, thick] (0,3)--(0,-3);
\draw[gray, thick] (0,-3)--(0,3);
\draw[gray, thick] (2,-2) -- (0,0);
\draw[gray, thick] (-3,0) -- (3,0);
\draw[gray, thick] (0,0) -- (2,2);
\draw[gray, thick] (2,2) -- (0,0);
\draw[gray,thick] (0,0)--(45:1.5cm);
\draw[gray, thick,<->] (1,0) arc  (0:45:1);  
\node at (1.25,.75){$\omega$};
\draw[gray,thick,-<] (0,0)--(1,1);
\draw[gray,thick,->] (0,0)--(1,-1);
%pic["$\alpha$", draw=orange, <->, angle eccentricity=1.2, angle radius=1cm] {angle=(1,0)--(0,0)--(1,1)};
\filldraw[black] (0,0) circle (2pt) node[anchor=west]{};
\end{tikzpicture}
\end{center}
\caption{} 
\label{fig:S}
\end{figure}

\begin{definition}
We say that $\textbf{A}$ admits a joint bounded
$H^\infty\big(\prod_{i=1}^n\Sigma_{\theta_i}\big)$ functional calculus, if for all $f\in H_0^\infty\big(\prod_{i=1}^n\Sigma_{\theta_i}\big),$ there exists a constant $C>0$ (independent of $f$), such that
\begin{equation*}
\|f(\textbf{A})\|_{X\to X}\leq C\|f\|_{\infty,\prod_{i=1}^n\Sigma_{\theta_i}}.
\end{equation*} 
\end{definition}

\subsection{Ritt operators:}
We recommend \cite{LL}, \cite{LeM2} and \cite{RA2} references therein  for details about Ritt operators. 
\begin{definition}[Ritt operator]\label{RIT}
For any $\gamma\in(0,\frac{\pi}{2})$, let $\mathcal{B}_\gamma$ (Stolz domain of angle 
$\gamma$) be the interior of the convex hull of $1$ and the disc $D(0,\sin\gamma)$ (FIGURE \ref{fig:R}). An operator $T:X\to X$ is said to be a Ritt  operator of type $\alpha\in(0,\frac{\pi}{2})$ if

\item[(i)] $\sigma(T)\subseteq\overline{\mathcal{B}_\alpha}.$ (see FIGURE \ref{fig:R})\label{ll}
\item[(ii)]For any $\beta\in(\alpha,\frac{\pi}{2}),$ the set
$\{(1-\lambda)R(\lambda,T):\lambda\in\mathbb{C}\setminus\overline{\mathcal{B}_\beta}\}$ is bounded.
\end{definition}

%The condition (ii) in the above definition is often referred as the Ritt resolvent condition. Following result provides an alternate definition of Ritt operators.
%\begin{proposition}[\cite{NJ}]\label{alternate}
%Let $T:X\to X$ be a bounded operator. Then, $T$ is a Ritt  operator if and only if
%\begin{itemize}
%\item[(1)]The set $\{T^n:n\geq 1\}$ is bounded.
%\item[(2)]The set $\{n(T^n-T^{n-1}):n\geq 1\}$ is bounded.
%\end{itemize}
%\end{proposition}
\begin{figure}\label{FIG2}
\begin{center}
\begin{tikzpicture}
\draw[black, thick] (0,0) -- (-1.435,1);
\draw[black,thick] (0,0)-- (-1.435,-1);
\draw [dashed] (-1.435,1) -- (-3.731,2.5);
\draw [dashed]  (-1.435,-1) -- (-3.731,-2.5);
\draw [black,thick] (0,0)--(-2.3,0);

\draw [gray!10,fill= gray!10] (0,0) -- (-1.435,1)  arc  (120:150:3.059225) -- cycle;
\draw [gray!10,fill= gray!10] (0,0) -- (-1.435,-1)  arc  (210:180:3.059225) -- cycle;
\draw[gray, thick] (0,0) arc  (0:360:2.3);
\draw[gray!10, thick,fill=gray!10] (-1,0) arc  (0:360:1.2);
\draw [black,thick] (0,0)--(-2.3,0);
\node at (190:2.40cm) {0};
\node at (165:1.9cm) {$\mathcal{B}_\gamma$};
\draw [black] (0,0) -- (-.3587,.25)  arc  (120:165:.4) -- cycle;
\node at (165:.7cm) {$\gamma$};
\node at (0.3,0) {1};
\end{tikzpicture}
\end{center}
\caption{} 
\label{fig:R}
\end{figure}
Let $\gamma_i\in(0,\frac{\pi}{2}),\ 1\leq i\leq n.$ Let $H_0^\infty\big(\prod_{i=1}^n\mathcal{B}_{\gamma_i}\big)$ be the set of all bounded holomorphic functions $\phi:\prod_{i=1}^n\mathcal{B}_{\gamma_i}\to
\mathbb{C},$ such that 
\begin{equation*}
|\phi(\lambda_1,\dots,\lambda_n)|\leq c\prod_{i=1}^n|1-\lambda_i|^s,\ \text{for}\ (\lambda_1,\dots,\lambda_n)\in\prod_{i=1}^n\mathcal{B}_{\gamma_i}
\end{equation*}
for some constants $c,s>0,$ depending only on $\phi.$ 
%Let us define the contours 
%\begin{eqnarray*}
%\Gamma_{\mathcal{B}_{\beta}}^1 &=& 1+[\sin\beta\exp(i(\frac{\pi}{2}-\beta))-1],\ 0\leq t\leq 1\\\nonumber
%\Gamma_{\mathcal{B}_{\beta}}^2&=&\sin\beta\exp\theta,\ \frac{\pi}{2}-\beta\leq\theta\leq\frac{3\pi}{2}+\beta\\\nonumber
%\Gamma_{\mathcal{B}_{\beta}}^3&=&\sin\beta\exp(i(\frac{3\pi}{2}+\beta))+t[1-\sin\beta\exp(i\frac{3\pi}{2})],\ 0\leq t\leq 1.
%\end{eqnarray*}
%Then the contour defined by $\Gamma_{\mathcal{B}_{\beta}}:=\Gamma_{\mathcal{B}_{\beta}}^1\oplus\Gamma_{\mathcal{B}_{\beta}}^2\oplus\Gamma_{\mathcal{B}_{\beta}}^3
%$ 
Let $\Gamma_{\mathcal{B}_{\beta}}$ denote the boundary of $\mathcal{B}_\beta$ oriented counterclockwise.
Suppose $\textbf{T}=(T_1,\dots,T_n)$ is a commuting tuple of Ritt operators such that each $T_i$ is a Ritt operator of type $\alpha_i\in(0,\frac{\pi}{2})$ for $1\leq i\leq n.$ Suppose $\beta_i\in(\alpha_i,\gamma_i),$ 
$1\leq i\leq n.$ We define
\begin{equation}\label{bal}
\phi(\textbf{T}):=\Big(\frac{1}{2\pi i}\Big)^n\int_{\prod_{i=1}^n\Gamma_{\mathcal{B}_{\beta_i}}}\phi(\lambda_1,\dots,\lambda_n)\prod_{i=1}^nR(\lambda_i,A_i)d\lambda_i.
\end{equation}
The above integral is absolutely convergent due to adequate decay on the resolvent operators. By Cauchy's theorem it is independent of $\beta_i,$ $1\leq i\leq n.$

\begin{definition}
We say that $\textbf{T}$ admits a joint bounded 
$H^\infty\big(\prod_{i=1}^n\mathcal{B}_{\gamma_i}\big)$ functional calculus if 
\begin{equation}
\|\phi(\textbf{T})\|_{X\to X}\leq C\|\phi\|_{\infty,\prod_{i=1}^n\mathcal{B}_{\gamma_i}},
\label{FC}
\end{equation} 
for some positive constant $C>0,$ which is independent of $\phi$ and $\phi\in H_0^\infty(\prod_{i=1}^n\mathcal{B}_{\gamma_i}).$
\end{definition}

%For our purpose it is enough to check (\ref{FC}) for polynomials. 
%
%\begin{lemma}\label{polyno}Let $\textbf{T}=(T_1,\dots,T_n)$. Then $(T_{i_1},\dots,T_{i_k})$ admits a joint bounded $H^\infty$ functional calculus for all $1\leq i_1<\dots<i_k\leq n$ if and only if there exists a constant $K>0$ such that for any $P\in\mathbb{C}[Z_1,\ldots,Z_n]$ \[\|P(\textbf{T})\|_{X\to X}\leq K\|P\|_{\infty, \prod_{i=1}^n\mathcal{B}_{\gamma_i}}\]
%\end{lemma}
%The above lemma can be proved exactly as the single variable case in \cite{LeM2}.

In the spirit of \cite{AFL} and \cite{ALM} we investigate necessary and sufficient conditions for commuting tuple of Ritt operators to admit a joint bounded functional calculus. We prove the following useful transfer result.
\begin{theorem}[Transfer principle]\label{trans}
 Suppose $\mathbf{T}=(T_1,\dots,T_n)$ is a commuting tuple of Ritt operators on $X.$ Let us denote the sectorial operators $A_i=I_X-T_i$ for $1\leq i\leq n.$ Then, the following are 
equivalent.
\begin{itemize}
\item[1.]There exist $\gamma_i\in(0,\frac{\pi}{2}),$ $1\leq i\leq n,$ such that for all $1\leq k\leq n,$ and $1\leq i_1<\dots<i_k\leq n,$ the tuples $(T_{i_1},\dots,T_{i_k})$ admit a joint bounded $H^\infty(
{\mathcal{B}}_{
{\gamma_{i_1}}}\times\cdots\times{\mathcal{B}}_{{\gamma_{i_k}}})$ functional calculus
\item[2.]There exist $\theta_i\in(0,\frac{\pi}{2}),$ $1\leq i\leq n,$ such that for all $1\leq k\leq n,$ and $1\leq i_1<\dots<i_k\leq n,$ the tuples $(A_{i_1},\dots,A_{i_k})$ admit a joint bounded $H^\infty(\Sigma_{\theta_{i_1}}\times\dots\times\Sigma_{\theta_{i_k}})$ functional calculus.
\end{itemize}
\end{theorem}
The above mentioned result is a mulitivariable generalisation of a similar transfer principle, proved in \cite{LeM2} and can be exploited to transfer known results about sectorial operators to Ritt operators.

The existence of bounded functional calculus is deeply connected to the notion of dilation. For instance, if $T$ is a contraction on a Hilbert space, then the classical result of Sz-Nagy and Foias (see \cite{SF}, Chapter 1) says that $T$ has a unitary dilation. Ando \cite{AT} showed that any commuting couple of contractions on a Hilbert space admits a joint unitary dilation on a common Hilbert space. Nagy-Foias and Ando's dilation theorem can be used to obtain von Neumann inequality in one and two variables respectively. However, there are examples \cite{GRA}, \cite{VA} of three commuting contractions on a Hilbert space which fail to have a joint unitary dilation. Moving out of the realm of Hilbert spaces, Akocglu and Sucheston \cite{AS} proved that any positive contraction on an $L^p$-space has a onto isometric dilation, $1<p\neq 2<\infty,$ which in turn shows that this class of contractions satisfy the Matsaev's conjecture. We refer \cite{CRW}, \cite{CRRW} and \cite{PEL} and references therein for more information in this direction and \cite{RA1} for multivariate generalisations. In \cite{AFL} and \cite{ALM}, the authors provided a characterisation of bounded functional calculus for Ritt operators in terms of loose dilation, which we define below.
\begin{definition} Let $1<p\neq 2<\infty.$ Suppose $\mathbf{T}=(T_1,\dots,T_n)$ be a commuting tuple of bounded operators on $L^p(\Omega)$. We say that the $n$-tuple $\mathbf{T}$ admits a joint isometric loose
dilation, if there exists a measure space $\Omega^\prime,$ a commuting tuple of onto isometries $\mathbf{U}=(U_1,\dots,U_n),$ on $L^p(\Omega^\prime),$ 
%with the property that the set $\{U_1^{i_1}\cdots U_n^{i_n}:i_1,\dots,i_n\in\mathbb{Z}\}$ is bounded, 
together with two bounded operators $\mathcal{Q}
:L^p(\Omega^\prime)\to L^p(\Omega)$ and $\mathcal{J}:L^p(\Omega)\to L^p(\Omega^\prime),$ such that $$T_1^{i_1}\cdots T_n^{i_n}=\mathcal{Q}U_1^{i_1}\cdots U_n^{i_n}\mathcal{J}$$ for all $i_1,\dots,i_n\in\mathbb{N}_0.$ 
%\begin{displaymath}
%\xymatrix{
%L^p(\Omega)\ar[rr]^{T_1^{i_1}\cdots T_n^{i_n}}  \ar[d]^{\mathcal{J}}  &&
%L^p(\Omega) \\
%L^p(\Omega') \ar[rr]^{U_1^{i_1}\cdots U_n^{i_n}}&& L^p(\Omega')\ar[u]^{\mathcal{Q}}
%}.
%\end{displaymath}
\end{definition}
In the case of Hilbert space and for any commuting tuple  $\mathbf{T}=(T_1,\dots,T_n)$  of bounded operators on $\mathcal{H},$ we say that the $n$-tuple $\mathbf{T}$ admits a joint isometric loose
dilation, if there exists a Hilbert space space $\mathcal{K},$ a commuting tuple of onto isometries $\mathbf{U}=(U_1,\dots,U_n),$ on $\mathcal{K},$ 
%with the property that the set $\{U_1^{i_1}\cdots U_n^{i_n}:i_1,\dots,i_n\in\mathbb{Z}\}$ is bounded, 
together with two bounded operators $\mathcal{Q}
:\mathcal{K}\to \mathcal{H}$ and $\mathcal{J}:\mathcal{H}\to \mathcal{K},$ such that \[T_1^{i_1}\cdots T_n^{i_n}=\mathcal{Q}U_1^{i_1}\cdots U_n^{i_n}\mathcal{J}\] for all $i_1,\dots,i_n\in\mathbb{N}_0.$

%We recommend the readers \cite{AK}, \cite{AFL}, \cite{LT} and \cite{PI4} for notions and results from Banach space theory.
 We prove the following theorem, which can be realised as a multivariate generalisation of Theorem (4.1) of \cite{AFL}.

\begin{theorem}\label{DIAL}
Let $1<p<\infty.$ Let $X$ be a reflexive Banach space such that both $X$ and $X^*$ have finite cotype. Let $\mathbf{T}=(T_1,\dots,T_n)$ be a commuting tuple of bounded operators on $X$ such that each $T_i$ is Ritt operator and admits a bounded $H^\infty$-functional calculus. Then, there exists a measure space $\Omega,$ a commuting tuple of isometric isomorphisms $\mathbf{U}=(U_1,\dots,U_n)$ 
on $L^p(\Omega,X),$ together with two bounded operators $\mathcal{Q}:L^p(\Omega,X)\to X$ and $\mathcal{J}:X\to L^p(\Omega,X),$ such that 
\begin{equation*}\label{DIA}
T_1^{i_1}\cdots T_n^{i_n}=\mathcal{Q}U_1^{i_1}\cdots U_n^{i_n}
\mathcal{J}, \text{for all}\ \ i_1,\dots,i_n\in\mathbb{N}_0.
\end{equation*}
%\begin{displaymath}
%\xymatrix{
%X\ar[rr]^{T_1^{i_1}\cdots T_n^{i_n}}  \ar[d]^{\mathcal{J}}  &&
%X \\
%L^p(\Omega,X) \ar[rr]^{U_1^{i_1}\cdots U_n^{i_n}}&& L^p(\Omega,X)\ar[u]^{\mathcal{Q}}
%}
%\end{displaymath}

In addition, we have
\begin{itemize}
\item[1.]If $X$ is an ordered Banach space, the maps $U_j$'s can be chosen as positive operators, $1\leq j\leq n.$
\item[2.] If $X$ is a closed subspace of an $L^p$-space, then $L^p(\Omega,X)$ is again a closed subspace of an $L^p$-space and the maps $U_j$'s can be chosen to be restrictions of a commuting
tuple of positive isometric isomorphisms on an $L^p$-space, $1\leq j\leq n.$
\item[3.]If $X$ is an $SQ_p$ space, then the Banach space $L^p(\Omega,X)$ is again an $SQ_p$ space and the maps $U_j$'s can be chosen to be compressionss of a commuting tuple of positive isometric 
isomorphisms on an $L^p$-space, $1\leq j\leq n.$
\end{itemize}
\end{theorem}
As far as we know there is no Ando like result for commuting positive operators on $L^p$-spaces.  However, the following result states that  if $(T_1,T_2)$ are commuting tuple of positive contractions and one of them is Ritt and admits a bounded $H^\infty$-functional calculus, the tuple admits a joint isometric loose dilation. In general for $n$-tuple we have the following result whose proof will be similar to the proof of Theorem \eqref{DIAL}.
 \begin{theorem}
Let $1<p<\infty.$ Let $\mathbf{T}=(T_1,\dots,T_n)$ be a commuting tuple of bounded operators on $L^p(\Omega)$ such that 
\begin{enumerate}
\item[1.] $T_1$ admits a loose dilation,
\item[2.] Each $T_i$ is Ritt operator and $T_i$ admits a bounded $H^\infty$-functional calculus for $2\leq i\leq n.$
\end{enumerate}
 Then $\mathbf{T}$ admits a joint isometric loose dilation.
 \end{theorem}
The concept of loose dilation leads to the notion of $p$-polynomially boundedness and $p$-completely polynomially boundedness (see \cite{ALM}). We need the following definitions.

For $1\leq j\leq n$, we define the $j$-th left shift operator on $\ell_p(\mathbb{Z}^n)$ as 
\[
              (S_j(a))_{i_1,\dots,i_j,\dots,i_n}:=a_{i_1,\dots,i_j-1,\dots,i_n},\ a\in\ell_p(\mathbb{Z}^n).
              \]

\begin{definition}[Jointly $p$-polynomially bounded] Let $1<p<\infty.$ Let $\mathbf{T}=(T_1,\dots,T_n)$ be a commuting tuple of bounded operators on $X$. We say that the $n$-tuple $\mathbf{T}$ is jointly 
$p$-polynomially bounded, if there exists a $K>0$ (independent of $P$), such that
\begin{equation*}
\|P(\mathbf{T})\|_{X\to X}\leq K\|P(\mathcal{S})\|_{\ell_p(\mathbb{Z}^n)\to\ell_p(\mathbb{Z}^n)},
\end{equation*}
for any polynomial $P\in\mathbb{C}[Z_1,\dots,Z_n],$ where  $\mathcal{S}:=(S_1,\c\ldots,S_n)$ is the commuting tuple of left shift operators on $\ell_p(\mathbb{Z}^n).$
\end{definition}
\begin{definition}[Jointly $p$-completely polynomially bounded]Let $1<p<\infty$ and  $\textbf{T}=(T_1,\ldots,T_n)$ be a commuting tuple of bounded operators on $X.$ We say that $\textbf{T}$ is jointly $p$-completely polynomially bounded, if for all $N\in\mathbb{N},$ there exists a constant $C>0,$ such that \[\big\|(P_{i,j}(\textbf{T}))\big\|_{L^p([N],X)\to L^p([N],X)}\leq C\big\|(P_{i,j}(\mathcal{S}))\big\|_{L^p([N],\ell^p(\mathbb{Z}^n))\to L^p([N],\ell^p(\mathbb{Z}^n))},\]
where $P_{i,j}\in\mathbb{C}[Z_1,\ldots,Z_n],$ $1\leq i,j\leq N $  and $[N]=:\{1,\dots,N\}$ with counting measure.
\end{definition}
Note that, in the case of single variable and operators on Hilbert space, the above notions of $2$-polynomially boundedness and $2$-completely polynomially boundedness agree with the usual notions of polynomially boundedness and completely polynomially boundedness respectively. We refer \cite{PI1} for a detailed exposition regarding these concepts.

We require the important notion of $R$-boundedness. Let us consider the probability space $\Omega_0=\{\pm 1\}^{\mathbb{Z}}.$ For any integer $k\in\mathbb{Z},$ denote the $k$-th coordinate function as $\epsilon_k(\omega)=\omega_k,$ where $\omega=(\omega_j)_{j\in\mathbb{Z}}\in\Omega_0.$ The sequence of i.i.d. random variables $(\epsilon_k)_{k\geq 0}$ is called the Rademacher system on the
probability 
space $\Omega_0.$ For  $1\leq p<\infty$ we denote the Banach space $\text{Rad}_p(X)\subseteq L^p(\Omega_0,X)$ to be the closure of the set $\text{span}\{\epsilon_k\otimes\omega_k:k\in\mathbb{Z},x_k\in X\}$ 
in the Bochner space $L^p(\Omega_0,X).$  For $p=2$ we simply denote  $\text{Rad}(X)$.

Let $E\subseteq B(X)$ be a set of bounded operators in $X$. We say that $E$ is $R$-bounded provided, there exists a constant $C>0,$ such that for any finite sequence
$(T_k)_{k=0}^N$ of $E$ and a finite sequence $(x_k)_{k=0}^N$ of $X,$
\begin{equation}\label{R}
\Big\|\sum_{k=0}^N\epsilon_k\otimes T_k(x_k)\Big\|_{\text{Rad(X)}}\leq C\Big\|\sum_{k=0}^N\epsilon_k\otimes x_k\Big\|_{\text{Rad(X)}}.
\end{equation}
We also need the notions of \textbf{\textit{$R$-Ritt}} and \textbf{\textit{$R$-sectorial}} operators, which one obtains by replacing \textbf{\textit{bounded}} by \textbf{\textit{$R$-bounded}} in Definition \eqref{SEC} and Definition \eqref{RIT} respectively. We suggest  \cite{BO}, \cite{WE2} and references therein for more about $R$-boundedness and its applications.

Following result provides a characterisation of joint bounded functional calculus for commuting tuple of Ritt operators on $L^p$-spaces, for $1<p<\infty,$ where each of them admits a bounded functional calculus.

\begin{theorem}\label{CLASS}
Let $1<p\neq 2<\infty$ and $\mathbf{T}=(T_1,\dots,T_n)$ be a commuting tuple of Ritt operators on $L^p(\Omega).$ Then the following assertions are equivalent.
\begin{itemize}
\item[1.]There exist $\gamma_i\in(0,\frac{\pi}{2}),$ $1\leq i\leq n,$ such that for all $1\leq k\leq n,$ and $1\leq i_1<\dots<i_k\leq n,$ the tuples $(T_{i_1},\dots,T_{i_k})$ admit a joint bounded $H^\infty(
{\mathcal{B}}_{
{\gamma_{i_1}}}\times\cdots\times{\mathcal{B}}_{{\gamma_{i_k}}})$ functional calculus
\item[2.]Each $T_i,$ $1\leq i\leq n$ is $R$-Ritt and $\mathbf{T}$ admits a joint isometric loose dilation.
\item[3.]Each $T_i,$ $1\leq i\leq n$ is $R$-Ritt and $\mathbf{T}$ is jointly $p$-completely polynomially bounded.
\item[4.]Each $T_i,$ $1\leq i\leq n$ is $R$-Ritt and $\mathbf{T}$ is jointly $p$-polynomially bounded.
\item[5.] Each $T_i,$ $1\leq i\leq n$ is $R$-Ritt and $I-T_i$ admits a bounded $H^\infty(\Sigma_{\theta_i})$ functional calculus for $\theta_i\in(0,\pi)$,  for each $1\leq i\leq n.$
\end{itemize}
\end{theorem}
For single Ritt operators we refer \cite{ALM}.
% In the last section, we prove  partial results similar to above, in the context of non-commutative $L^p$-spaces. For arbitrary Banach spaces we have the following result.
%
%\begin{theorem}Let $\mathbf{T}=(T_1,\dots,T_n)$ be a commuting tuple of bounded operators on $X,$ which is jointly $p$-polynomially bounded. Then $(I-T_1,\dots,I-T_n)$ is a commuting tuple of sectorial operators and  admits a joint bounded $H^\infty(\prod_{i=1}^n\Sigma_{\theta_i})$ functional calculus for  all $\theta_i\in(\frac{\pi}{2},\pi)$.
%
%\label{Banach_CLASS}
%\end{theorem}

\subsection{Similarity problem:}
If an operator $T$ on a Hilbert space is similar to a contraction then it is polynomially bounded. The converse was a long--standing open problem for many years (see \cite{HA}). Finally, Pisier \cite{PI2} in 1997 constructed an example to show that the converse  does not hold. Essentially, he produced an operator which is polynomially bounded but not completely polynomially bounded (In between Paulsen \cite{PA2} had shown that $T$ is similar to contraction if and only if it is completely polynomially bounded). For multivariable case, it is natural to ask if for any commuting tuple of bounded operators $\textbf{T}=(T_1,\ldots,T_n),$ $T_i\in B(\mathcal{H}),$ $1\leq i\leq n,$ such that each $T_i$ is similar to a contraction will automatically imply that $T_i$'s are jointly similar to a commuting tuple of contractions, i.e., if there exists an invertible map $S$ such that $(ST_1S^{-1},\dots,ST_nS^{-1})$ is a commuting tuple of contractions ? Again, Pisier \cite{PI3} (also see \cite{PE}) answered it negatively by showing existence of two commuting contractions $(T_1,T_2)$  on a Hilbert space, $T_i$ being similar to a contraction, $1\leq i\leq 2,$ but $T_1T_2$ is not even polynomially bounded. In this context, one can also ask for examples of class of operators, for which the joint similarity problem has a positive solution. It is known that the Halmos's similarity problem has a positive answer in the class of Ritt operators \cite{DE}, \cite{LeM2}, \cite{LeM3}. Following result asserts an affirmative answer for joint similarity problem in the class of Ritt operators.

\begin{theorem}\label{JRS}
Let $\mathbf{T}=(T_1,\dots,T_n)$ be a commuting tuple of Ritt operators on a Hilbert space $\mathcal{H}.$ Then, the following assertions are equivalent.
\begin{itemize}
\item[1.] Each $T_i$ is similar to a contraction, $1\leq i\leq n.$
\item[2.] For all $1\leq k\leq n,$ and $1\leq i_1<\dots<i_k\leq n,$ the tuples $(T_{i_1},\dots,T_{i_k})$ admit a joint bounded $H^\infty$-functional calculus.
\item[3.]The tuple $\mathbf{T}=(T_1,\dots,T_n)$ is jointly similar to a commuting $n$-tuple of contractions.
\end{itemize}
\end{theorem}
The above theorem can be interpreted as a Hilbert space variant of Theorem \eqref{CLASS}. Also, it is natural to ask for examples of couple of commuting operators $(T_1,T_2)$ on general Banach spaces, such that each of which is $p$-polynomially bounded (resp. completely $p$-polynomially bounded) but not jointly $p$-polynomially bounded (resp. jointly $p$-completely polynomially bounded). We do not have such example. However, in view of Theorem \eqref{CLASS} we have a positive answer for Ritt operators on $L^p$-spaces which are positive contractions, $1<p\neq 2<\infty.$

\begin{center}\section{\bf Proof of the results:}\end{center}
 
\begin{center}{\bf Proof of Theorem \eqref{trans}:}
\end{center}

Let us denote $\Delta_\gamma:=1-\mathcal{B}_{\gamma}$ for $\gamma\in(0,\frac{\pi}{2}).$ Fix $\gamma_i\in(0,\frac{\pi}{2})$ and $\beta_i\in(0,\gamma_i),$ for $1\leq i\leq n.$ Let $\alpha_i<\theta<\beta_i.$  Let $\Gamma_1^i$ be the contour connecting $\cos\beta_i\exp(i\beta_i)$ to $0$ and $0$ to $\cos\beta_i\exp(-i\beta_i)$ via straight lines, and $\Gamma_2^i$ be the contour going from $\cos\beta\exp(-i\beta_i)$ to $\cos\beta\exp(i\beta_i)$ along the circle with centre at $1$ and radius $\sin\beta$ in the counterclockwise direction (see FIGURE \ref{fig:JO}). Define the contour $\Gamma^i:=\Gamma_1^i+\Gamma_2^i.$
Let $f\in H^\infty(\prod_{i=1}^n\Delta_{\gamma_i})$ satisfying $|f(z_1,\ldots,z_n)|\leq c_1\prod_{i=1}^n|z_i|^s$ for some $s>0$. 
\begin{figure}\label{FIG3}
\begin{center}
\begin{tikzpicture}
\draw [very thick,black,->] (0,-5)--(0,5);
\draw [very thick,black,->] (-1,0)--(10,0);
\draw [black,thick] (1,2)--(2,4);
\node at (2.1,4.3){$\gamma_i$};
\draw [black,thick] (1,-2)--(2,-4);
\draw [ultra thick,black,dashed] (2.3,0) circle(2.3cm);
\draw [black,thick] (0,0)--(3,4);
\node at (3.2,4.3) {$\beta_i$};
\draw [black!100,very thick] (0,0)--(1,4/3);
\draw [black,thick] (0,0)--(3,-4);
\draw [black!100,very thick] (0,0)--(1,-4/3);
\def\centerarc[#1](#2)(#3:#4:#5)% Syntax: [draw options] (center) (initial angle:final angle:radius)
    { \draw[#1] ($(#2)+({#5*cos(#3)},{#5*sin(#3)})$) arc (#3:#4:#5); };
 
 \centerarc[black,very thick](2.3,0)(132:-132:1.85);
 \draw [black,thick] (0,0)--(4,4);
 \node at (4.1,4.3) {$\theta$};
 \draw [black,thick] (0,0)--(6,4);
 \node at (6.1,4.3) {$\alpha_i$};
 \draw[black,fill=gray!10] plot [smooth] coordinates{ (0,0) (.5,.2) (1,.5) (1.5,.55) (2,.8) (2.5,.5) (3,.5) (3.2,0) (3,-.5)(2.5,-.5) (2,-.8) (1.5,-.55) (1,-.5) (.5,-.2) (0,0)};
 \draw [very thick,black,->] (-1,0)--(10,0);
 \node[circle,draw=black, fill=black, inner sep=0pt,minimum size=5.2pt] (b) at (1,4/3) {};
 \node[circle,draw=black, fill=black, inner sep=0pt,minimum size=5.2pt] (b) at (1,-4/3) {};
 \draw [thick,black,->] (1,4/3)--(.5,2/3);
 \draw [thick,black,->] (0,0)--(.5,-2/3);
 \node at (.9,-.8) {$\Gamma_1^i$};
 \node at (2.8,-1.5) {$\Gamma_2^i$};
 \node at (1.9,.3) {$\sigma (A)$}; 
 \centerarc[black,very thick,->](2.3,0)(-132:-90:1.85);
 \draw [black,,->] plot[smooth] coordinates{(-.5,.9)(.1,1.2)(.9,4/3)};
 \node at (-1.25,.7) {$cos(\beta)e^{i\beta}$};
 \node at (-.25,-.2) {$0$};
 \draw [black,,->] plot[smooth] coordinates{(5,-2.5)(4.3,-2)(3.5,-1.75)};
\node at (5.19,-2.6) {$\Delta_{\gamma}$};
\end{tikzpicture}
\end{center}
\caption{} 
\label{fig:JO}
\end{figure}
For $j_k\in\{1,2\},$ $1\leq k\leq n,$ let us define (on an appropriate domain)
\begin{eqnarray*}\label{decompo}
f_{j_1,\ldots,j_n}(z_1,\ldots,z_n):=\Big(\frac{1}{2\pi i}\Big)^n\int_{\Gamma_{j_1}^1}\cdots\int_{\Gamma_{j_n}^n}\frac{f(\lambda_1,\ldots,\lambda_n)}
{\prod_{i=1}^n(\lambda_i-z_i)
}d\lambda_1\cdots d\lambda_n.
\end{eqnarray*} 
Observe that, $\ \text{for}\ (z_1,\dots,z_n)\in\prod_{i=1}^n\mathcal{B}_{\beta_i}$ by Cauchy's theorem, 
\begin{equation}\label{COCHY}
f(z_1,\dots,z_n)=\sum_{1\leq j_1,\dots,j_n\leq 2}f_{j_1,\dots,j_n}(z_1,\dots,z_n).
\end{equation}
Let us fix a tuple $(j_1,\ldots,j_n),$ where not all $j_k$'s take the same value. Define the following affine maps on $\mathbb{C}^n$ as
\begin{eqnarray*}\label{Se}
\mathcal{P}_i\mathfrak{F}(j_1,\ldots,j_n)(z_1,\ldots,z_n):=
\begin{cases} z_i,& j_i=1\\
        0,&j_i\neq 1,
         \end{cases}
\end{eqnarray*}
where $\mathcal{P}_i:\mathbb{C}^n\to\mathbb{C}^n$ is the $i$-th projection map, for $1\leq i\leq n.$  Define  
\begin{eqnarray*}
F_{j_1,\ldots,j_n}(z_1,\ldots,z_n)&:=&\frac{(f_{j_1,\ldots,j_n}(\mathfrak{F}(j_1,\ldots,j_n))(z_1,\ldots,z_n)}{\prod_{\{i,j_i=2\}}(1+z_i)},\\
F_{1,\ldots,1}(z_1,\cdots,z_n)&:=&f_{1,\ldots,1}(z_1,\cdots,z_n),\\
F_{2,\ldots,2}(z_1,\cdots,z_n)&:=&\frac{f_{2,\ldots,2}(0,\ldots,0)}{\prod_{i=1}^n(1+z_i)}.
\end{eqnarray*}

First, we will prove two lemmas, which are essential for the proof of Theorem \eqref{trans}. We need the following single variable version of Lemma \eqref{BESTLEMMA} which was part of  Proposition (4.1)  in \cite{LeM2}.
\begin{lemma}[\cite{LeM2}] \label{SINGLEVARIABLE}
Let $0<\theta<\beta_1<\gamma_1<\frac{\pi}{2}.$ Let $f\in H^\infty(\Delta_{\gamma_1})$ satisfying $|f(z_1)|\leq c_1|z_1|^s$ for some $c,\ s>0$. Let us define \[g(z_1)=F_1(z_1)+\frac{1}{1+z_1}F_2(0),\] where \[F_j(z_1)=\frac{1}{2\pi i}\int_{\Gamma_j^1}\frac{f(\lambda_1)}{\lambda_1-z_1}d\lambda_1, \ \text{for}\ j=1,2.\] Then $g\in H_0^\infty(\Sigma_\theta)$ and $\|g\|_{\infty,\Sigma_\theta}\leq C\|f\|_{\infty,\Delta_{\gamma_1}}.$
\end{lemma}
%
%\begin{proof}
%Note that\[F_1(z_1)=\frac{1}{2\pi iz_1}\int_{\Gamma_1^1}\frac{f(\lambda_1)}{\frac{\lambda_1}{z_1}-1}d\lambda_1.\] Therefore, for $|z_1|\to\infty$ we have, $z_1F_1(z_1)$ is bounded on $\Sigma_{\theta}.$ We rewrite $g$ as \begin{equation}\label{OX} g(z_1)=f(z_1)+(F_2(0)-F_2(z_1))-F_2(0)\frac{z_1}{1+z_1},\ z_1\in\Delta_{\theta}.
%\end{equation}
%It is easy to see that $|F_2(z_1)-F_2(0)|\leq c_1|z_1|,\; {\rm for}\; z_1\in \Delta_{\theta}.$ Therefore, by \eqref{OX}, we obtain that $|g(z_1)|\leq c_2|z_1|^{s} \; {\rm for}\; z_1\in\Delta_{\theta}$ for some $s>0.$ Thus, the first part of the lemma is proved.
%
%For the second part notice that $d(\Gamma_1^1,\Sigma_\theta\setminus\Delta_\theta)>0$. Therefore, we deduce the elementary estimate\[\|F_1\|_{\infty,\Sigma_\theta\setminus\Delta_\theta}\leq C_1\|f\|_{\infty,\Delta_{\gamma_1}}\] where the constant $C_1>0$ is independent of $f.$ Again $d(\Gamma_2^1,\Delta_\theta)>0$ gives \[\|F_2\|_{\infty,\Delta_\theta}\leq C_2\|f\|_{\infty,\Delta_{\gamma_1}}\] and $C_2>0$ is independent of $f.$ For $z\in\Delta_{\beta_1}$, we have $f(z_1)=F_1(z_1)+F_2(z_1)$ by Cauchy's theorem. Therefore, one can easily deduce the estimate \[\|F_1\|_{\infty,\Sigma_\theta}\leq C\|f\|_{\infty,\Delta_{\gamma_1}}\] for some positive constant $C$, which is independent of $f.$ From this one readily establishes the last part of the lemma.
%\end{proof}

Throughout this section we will follow  notations as in the beginning of this section.
\begin{lemma}\label{BESTLEMMA}
Let us define the following auxiliary function
\begin{equation}{\label{eqg}}
g(z_1,\ldots,z_n)=\sum_{1\leq j_1,\ldots,j_n\leq 2}F_{j_1,\ldots,j_n}(z_1,\ldots,z_n).
\end{equation}
Let $\theta\in(0,\beta_i),$ $1\leq i\leq n.$
Then $g\in H_0^\infty(\Sigma_{\theta}^n).$ 
\end{lemma}
\begin{proof}
We will give the proof for $n=2$  to avoid notational complexity. For general $n$ the proof is in the same spirit.

It is clear that $g$ is holomorphic on appropriate domain. To establish the lemma, we need to consider the following cases.

{\bf Case 1:} Both $|z_1|$ and $|z_2|$ are large.  We use triangle inequality in \eqref{eqg} to obtain the following estimate on $g.$
\begin{eqnarray*}
|g(z_1,z_2)| &\leq& \frac{1}{4\pi^2|z_1||z_2|}\left| \int_{\Gamma_1^1}\int_{\Gamma_1^2} \frac{f(\lambda_1,\lambda_2)}{(\lambda_1/z_1-1)(\lambda_2/z_2-1)}d\lambda_1 d\lambda_2\right|\\
&& + \frac{1}{4\pi^2|z_1||1+z_2|} \left|\int_{\Gamma_1^1}\int_{\Gamma_2^2} \frac{f(\lambda_1,\lambda_2)}{(\lambda_1/z_1-1)\lambda_2}d\lambda_1 d\lambda_2\right|\\
&&+ \frac{1}{4\pi^2|z_2||1+z_1|}\left|\int_{\Gamma_2^1}\int_{\Gamma_1^2} \frac{f(\lambda_1,\lambda_2)}{\lambda_1(\lambda_2/z_2-1)}d\lambda_1 d\lambda_2\right| \\
&&+\left| \frac{1}{(1+z_1)(1+z_2)} f_4(0,0)\right|.
\end{eqnarray*}
Therefore, in this case $|g(z_1,z_2)|\leq c_2|z_1|^{-1}|z_2|^{-1}$ for some constant $c_2>0$ which is independent of $z_1,$ $z_2$.

{\bf Case 2:} $|z_2|$ is large and  $|z_1|$ small.
We group $F_{1,1}$ and $F_{2,1}$ in \eqref{eqg} to obtain the following expression
\begin{equation}\label{hmmm}
 -\frac{1}{4\pi^2}\Big(\int_{\Gamma_1^2} \frac{1}{\lambda_2-z_2}\Big(\int_{\Gamma_1^1}\frac{f(\lambda_1,\lambda_2)}{(\lambda_1-z_1)}d\lambda_1 +\frac{1}{1+z_1}\int_{\Gamma_2^1} \frac{f(\lambda_1,\lambda_2)}{\lambda_1}d\lambda_1\Big)d\lambda_2\Big).
%&=&\frac{1}{2\pi i}\Big(\int_{\Gamma_1^2} \frac{1}{\lambda_2-z_2} \Big(f(z_1,\lambda_2)- \frac{z_1}{2\pi i}\int_{\Gamma_2^1}\frac{f(\lambda_1,\lambda_2)}{(\lambda_1-z_1)\lambda_1} d\lambda_1- \frac{1}{2\pi i}\frac{z_1}{1+z_1}\int_{\Gamma_2^1}\frac{f(\lambda_1,\lambda_2)}{\lambda_1} d\lambda_1\Big) d\lambda_2\Big)
\end{equation}
%Let us observe the following
%\begin{eqnarray*}
% f_{1,2}(z_1,0)+\frac{1}{(1+z_1)}f_{2,2}(0,0)&=& -\frac{1}{4\pi^2}\Big(  \int_{\Gamma_2^2}\frac{1}{\lambda_2}\Big(\int_{\Gamma_1^1} \frac{f(\lambda_1,\lambda_2)}{(\lambda_1-z_1)} d\lambda_1\\
% &&+\frac{1}{1+z_1}\int_{\Gamma_2^1} \frac{f(\lambda_1,\lambda_2)}{\lambda_1}d\lambda_1\Big)d\lambda_2\Big).
%\end{eqnarray*}
For a fixed $\lambda_2$ let us define, $\tilde f^{\lambda_2}(z)=f(z_1,\lambda_2)$. 
Applying Lemma \eqref{SINGLEVARIABLE} to the function $\widetilde f^{\lambda_2},$
the term inside the bracket of \eqref{hmmm} will produce a factor of $|z_1|^s$ for some $s>0$ and as $|z_2|$ is large we will get a factor of $\frac{1}{|z_2|}$. Hence, we have the estimate
$\left|F_{1,1}(z_1,z_2)+F_{2,1}(0,z_2)\right|\leq c_3 |z_1|^s|z_2|^{-1}.$
Similarly, we group $F_{1,2}$ and $F_{2,2}$ in \eqref{eqg} to obtain appropriate bound for $g$.
%\[|g(z_1,z_2)|\leq c_3|z_1||z_2|^{-1}.\]

{\bf Case 3:} $|z_1|$ is large and  $|z_2|$ small. This is similar to the previous case. 

{\bf Case 4:} Both $|z_1|$ and $|z_2|$ are small. We apply Cauchy's theorem to obtain the following identity
\begin{eqnarray*}
g(z_1,z_2)&=& f(z_1,z_2)-\left(f_{1,2}(z_1,z_2)-\frac{1}{1+z_2} f_{1,2}(z_1,0)\right)\\
&& -\Big(f_{2,1}(z_1,z_2)-\frac{1}{1+z_1} f_{2,1}(0,z_2)\Big)\\& &-\Big(f_{2,2}(z_1,z_2)-\frac{1}{(1+z_1)(1+z_2)} f_{2,2}(0,0)\Big)\\
&=& I_1-I_2-I_3-I_4.
\end{eqnarray*}
Let us rewrite and notice that
\begin{eqnarray*}
I_2&=& (f_{1,2}(z_1,z_2)- f_{1,2}(z_1,0))+\frac{z_2}{1+z_2} f_{1,2}(z_1,0)\\
&=& \frac{1}{2\pi i}  \int_{\Gamma_2^2} \frac{z_2f(z_1,\lambda_2)}{\lambda_2(\lambda_2-z_2) } d\lambda_2+\frac{1}{4\pi^2}\int_{\Gamma_2^1}\int_{\Gamma_2^2} \frac{z_2f(\lambda_1,\lambda_2)}{\lambda_2(\lambda_1-z_1)(\lambda_2-z_2)}d\lambda_1 d\lambda_2\\&&+\frac{z_2}{1+z_2} f_{1,2}z_1,0).
\end{eqnarray*}
By an analogous computation, we derive 
\begin{eqnarray*}
I_3&=&  \frac{1}{2\pi i}  \int_{\Gamma_2^1} \frac{z_1f(\lambda_1,z_2)}{\lambda_1(\lambda_1-z_2)} d\lambda_1\\&&+\frac{1}{4\pi^2}\int_{\Gamma_2^1}\int_{\Gamma_2^2} \frac{z_1f(\lambda_1,\lambda_2)}{\lambda_1(\lambda_1-z_1)(\lambda_2-z_2)}d\lambda_1 d\lambda_2+\frac{z_1}{1+z_1} f_{2,1}(z_1,0).
\end{eqnarray*}

Now  we observe the following
\begin{eqnarray*}
&&f_{2,2}(z_1,z_2)-f_{2,2}(0,0)\\
&=&-\frac{1}{4\pi^2}\Big( \int_{\Gamma_2^1}\int_{\Gamma_2^2} \frac{z_2f(\lambda_1,\lambda_2)}{\lambda_2(\lambda_1-z_1)(\lambda_2-z_2)}d\lambda_1 d\lambda_2\\
&&+\int_{\Gamma_2^1}\int_{\Gamma_2^2} \frac{z_1f(\lambda_1,\lambda_2)}{\lambda_1(\lambda_1-z_1)(\lambda_2-z_2)}d\lambda_1 d\lambda_2\\
&&- z_1z_2 \int_{\Gamma_2^1}\int_{\Gamma_2^2} \frac{f(\lambda_1,\lambda_2)}{\lambda_2\lambda_1(\lambda_2-z_2)(\lambda_1-z_1)}d\lambda_1 d\lambda_2 \Big).
\end{eqnarray*}

Hence, by adding all the terms we get
\begin{eqnarray*}
g(z_1,z_2) &=&f(z_1,z_2) -\frac{1}{2\pi i}\left( \int_{\Gamma_2^2} \frac{z_2f(z_1,\lambda_2)}{\lambda_2(\lambda_2-z_2)} d\lambda_2  + \int_{\Gamma_2^1} \frac{z_1f(\lambda_1,z_2)}{\lambda_1(\lambda_1-z_1)} d\lambda_1\right)\\
&&-\frac{z_2}{1+z_2} \Big( f_{1,2}(z_1,0)+\frac{1}{(1+z_1)}f_{2,2}(0,0)\Big) \\
&&- \frac{z_1}{1+z_1} \Big(f_{2,1}(0,z_2) +\frac{1}{(1+z_2)}f_{2,2}(0,0)\Big)\\
 &&-z_1z_2 \Big(\frac{1}{4\pi^2}  \int_{\Gamma_2^1}\int_{\Gamma_2^2} \frac{f(\lambda_1,\lambda_2)}{\lambda_2\lambda_1(\lambda_2-z_2)(\lambda_1-z_1)}d\lambda_1 d\lambda_2\\
 &&\hspace*{1.3in}+ \frac{1}{(1+z_1)(1+z_2)}f_{2,2}(0,0)\Big).
\end{eqnarray*}
Since $|z_1|,|z_2|$ both small, we have for some $s_1,s_2>0,$
\begin{equation}\label{Dusra}
 \left| \int_{\Gamma_2^2} \frac{z_2f(z_1,\lambda_2)}{\lambda_2(\lambda_2-z_2)} d\lambda_2  + \int_{\Gamma_2^1} \frac{z_1f(\lambda_1,z_2)}{\lambda_1(\lambda_1-z_1)} d\lambda_1 \right|\leq K_1|z_1|^{s_1}|z_2|^{s_2},
\end{equation}
where $K_1>0$ is a constant which depends on $f.$
%Also
%\begin{eqnarray*}
% f_{1,2}(z_1,0)+\frac{1}{(1+z_1)}f_{2,2}(0,0)&=& -\frac{1}{4\pi^2}\Big(  \int_{\Gamma_2^2}\frac{1}{\lambda_2}\Big(\int_{\Gamma_1^1} \frac{f(\lambda_1,\lambda_2)}{(\lambda_1-z_1)} d\lambda_1\\
% &&+\frac{1}{1+z_1}\int_{\Gamma_2^1} \frac{f(\lambda_1,\lambda_2)}{\lambda_1}d\lambda_1\Big)d\lambda_2\Big).
%\end{eqnarray*}
%For a fixed $\lambda_2$ define, $\tilde f^{\lambda_2}(z)=f(z_1,\lambda_2)$. 
We apply Lemma \eqref{SINGLEVARIABLE} as we did to handle Case 2  to get a constant $K_2>0$ such that
\begin{eqnarray}\label{Tisra}
\Big|f_{1,2}(z_1,0)+\frac{1}{(1+z_1)}f_{2,2}(0,0)\Big|\leq K_2 |z_1|^s.
\end{eqnarray}
 Similarly, for some $K_3>0,$ we will obtain the bound 
\begin{eqnarray}\label{Chotha}
\left|f_{2,1}(0,z_2) +\frac{1}{(1+z_2)}f_{2,2}(0,0)\right|\leq K_3 |z_2|^s.
\end{eqnarray}
Therefore, in view of \eqref{Dusra}, \eqref{Tisra} and \eqref{Chotha}, we obtain that $g$ has the appropriate bound.

\end{proof}
\begin{proof}[\bf Proof of Theorem \eqref{trans}:] We only prove $(2)$ implies $(1).$  We also restrict ourselves for $n=2$, for general $n>2$ the proof can be done using induction. Let us assume that the tuple $\mathbf{A}=(A_1,A_2)$ admits a joint bounded $H^\infty(\prod_{i=1}^2\Sigma_{\theta_i})$ functional calculus  and $A_i$ admits a bounded $H^\infty(\Sigma_{\theta_i})$ functional calculus for
some $\theta_i\in(0,\frac{\pi}{2}),$ $1\leq i\leq 2.$ Suppose $\sigma(A_i)\subseteq\overline{\Delta_{\alpha_i}},$ $1\leq i\leq 2.$ We can assume that all $\theta_i$'s are equal and $\theta>\alpha_i$ for $1\leq i\leq 2.$ Now we choose $\beta_i$ and $\gamma_i$ in $(0,\frac{\pi}{2})$ so that $\theta<\beta_i<\gamma_i$ for $1\leq i\leq 2.$

Let $\phi\in H_0^\infty(\mathcal{B}_{\gamma_1}\times \mathcal{B}_{\gamma_1}).$ Define $f:\Delta_{\gamma_1}\times\Delta_{\gamma_2}\to\mathbb{C}$ to be \[f(z_1,z_2):=\phi(1-z_1,1-z_2).\] Now we define an auxilliary function $g$ as in \eqref{eqg}.
Since $g\in H_0^\infty(\Sigma_\theta\times\Sigma_\theta)$ by Lemma \eqref{BESTLEMMA} we can define $g(A_1,A_2)$ by \eqref{JO}. Let us define

\[f_{1,1}(A_1,A_2):=g(A_1,A_2)- (I+A_2)^{-1} f_{1,2}(A_1,0)-(I+A_1)^{-1}f_{2,1}(0,A_2)\]
\begin{equation}\label{ASI}
\vspace{.5mm}- (I+A_1)^{
-1}(I+A_2)^{-1} f_{2,2}(0,0),
\end{equation} where in \eqref{ASI}
$f_{1,2}(A_1,0)$ and $f_{2,1}(0,A_2)$ are defined as the following. Consider \[\widetilde{g}(z_1)= f_{1,2}(z_1,0)+\frac{1}{1+z_1}f_{2,2}(0,0).\] By Lemma \eqref{SINGLEVARIABLE} $\widetilde{g}\in H_0^\infty (\Sigma_\theta)$. Therefore, $\widetilde{g}(A_1)$ can be defined as in \eqref{JO}. Thus, we can define \[f_{1,2}(A_1,0):=\widetilde{g}(A_1)-(I+A_1)^{-1}f_{2,2}(0,0).\] We define $f_{2,1}(0,A_2)$ in a similar manner. Since $A_1$ admits a bounded functional calculus, we have the estimate $\|\widetilde{g}(A_1)\|\leq K^\prime\|\widetilde{g}\|_{\infty,\Sigma_\theta}.$ Therefore, it follows immediately from Lemma \eqref{SINGLEVARIABLE} that \[\|f_{1,2}(A_1,0)\|\leq K^{\prime\prime}\|\widetilde{g}\|_{\infty,\Sigma_\theta}\leq K_2\|f\|_{\infty, \Delta_{\gamma_1}\times \Delta_{\gamma_2}}.\] Similarly, we get the appropriate bound for $f_{2,1}$,\[\|f_{2,1}(0,A_2)\|\leq K_3\|f\|_{\infty, \Delta_{\gamma_1}\times \Delta_{\gamma_2}}.\] Note that the constants $K_2$ and $K_3$ are independent of $f.$ From \eqref{ASI}, we get
\begin{equation}
\|f_{1,1}(A_1,A_2)\|\leq K_4 \|f\|_{\infty, \Delta_{\gamma_1}\times \Delta_{\gamma_2}},
\end{equation}for some positive constant $K_4$ (independent of $f$). For a fixed $\lambda_2$ define $\tilde f^{\lambda_2}(z_1)=f(z_1,\lambda_2)$. Note that the function $$g^{\lambda_2}(z_1):=\frac{1}{2\pi i}\Big(\int_{\Gamma_1^1}\frac{\widetilde{f}^{\lambda_2}(\lambda_1)}{\lambda_1-z_1}d\lambda_1+\frac{1}{1+z_1}\int_{\Gamma_2^1}\frac{\widetilde{f}^{\lambda_2}(\lambda_1)}{\lambda_1}d\lambda_1\Big)$$ is in $H_0^\infty(\Sigma_\theta).$ Thus, we use Lemma \eqref{SINGLEVARIABLE} to have 
\begin{equation}\label{g^}
\|g^{\lambda_2}(A_1)\|\leq K_5\|f\|_{\infty, \Delta_{\gamma_1}\times\Delta_{\gamma_2}},
\end{equation}where, $K_5>0$ is independent of $f.$ Let us define
%Define \[\widetilde f^{\lambda_2}(A_1):=g^{\lambda_2}(A_1)-\frac{1}{2\pi i(1+A_1)}\int_{\Gamma_2^1}\frac{\widetilde{f}^{\lambda_2}(\lambda_1)}{\lambda_1}d\lambda_1.\] Clearly by \eqref{g^} we have $\|f^{\lambda_2}(A_1)\|\leq K^\prime \|f\|_{\infty, \Delta_{\gamma_1}\times\Delta_{\gamma_1}}.$
%Let us define
 \begin{eqnarray*}
 f_{1,2}(A_1,A_2)&:=&\frac{1}{2\pi i}\int_{\Gamma_2^2}g^{\lambda_2}(A_1)(\lambda_2-A_2)^{-1}d\lambda_2\\
 &&-\Big(\frac{1}{2\pi i}\Big)^2\int_{\Gamma_2^1}\int_{\Gamma_2^2}f(\lambda_1,\lambda_2)\lambda_1^{-1}(\lambda_2-A_2)^{-1}d\lambda_1d\lambda_2.
 \end{eqnarray*}
In above the second integral is defined as Dunford-Riesz functional calculus. Since $\sigma(A_i)\cap \Gamma_2^i=\emptyset,$ $i=1,2,$ one can observe that $\inf\limits_{\lambda_i\in \Gamma_2^i}\|\lambda_iI-A_i\|>0$ for $1\leq i\leq 2.$ Therefore, by \eqref{g^} we have an estimate 
\begin{equation}\label{ASIIII}
\|f_{1,2}(A_1,A_2)\|\leq K_6\|f\|_{\infty, \Delta_{\gamma_1}\times\Delta_{\gamma_2}}.
\end{equation}
By a similar way one can define $f_{2,1}(A_1,A_2)$
and obtain the estimate
\begin{equation}\label{ASIII}
\|f_{2,1}(A_1,A_2)\|\leq K_7\|f\|_{\infty, \Delta_{\gamma_1}\times\Delta_{\gamma_2}}.
\end{equation} 
 It is easy to see
\begin{equation}\label{ASII}
\|f_{2,2}(A_1,A_2)\|\leq K_8\|f\|_{\infty, \Delta_{\gamma_1}\times\Delta_{\gamma_2}}
\end{equation} Note that the constants $K_6,K_7,K_8>0$ are independent of $f$.
Hence by \eqref{ASI}, \eqref{ASIIII}, \eqref{ASIII} and \eqref{ASII} we have
\begin{eqnarray}\label{1771}
\|\phi(T_1,T_2)\| &=& \| f_{1,1}(A_1,A_2) +  f_{1,2}(A_1,A_2) + f_{2,1}(A_1,A_2) + f_{2,2}(A_1,A_2)\| \\\nonumber
&\leq &  K \|\phi\|_{\infty,\mathcal{B}_{\gamma_1}\times\mathcal{B}_{\gamma_2}},
\end{eqnarray}
 where $K>0$ is independent of $f.$
The proof is completed by \eqref{1771} and \cite{LeM2}.

\end{proof}
%\begin{remark} In \cite{LeM2}  a notion of joint quadratic functional calculus was introduced. An analogous transfer principle like Theorem \eqref{trans} can be proved in this general setting.
%\end{remark}
\begin{center}{\bf Proof of Theorem \eqref{DIAL}:}
\end{center}

Let $T$ be Ritt operator on $X.$ Since $I-T$ is a sectorial operator, one can define via functional calculus the fractional power $(I-T)^\alpha$ for any $\alpha>0.$ One defines the associated Littlewood-Paley square function \[\|x\|_{T,\alpha}:=\Big\|
\sum_{k=0}^\infty(k+1)^{\alpha-\frac{1}{2}}\epsilon_k\otimes T^k(I-T)^\alpha x\Big\|_{\text{Rad}(X)},\] where $x\in X.$ Note that it can happen that $\|x\|_{T,\alpha}=\infty$ for some $x\in X.$

Now we will give the proof of Theorem \eqref{DIAL} following the idea of Prof. Christian Le Merdy in a personal communication.

\begin{proof}[{\bf Proof of Theorem \eqref{DIAL}:}] We proceed by induction on the number of commuting operators. Let us define an operator $\mathsf{u}:L^p(\Omega_0)\to L^p(\Omega_0)$ as $$\mathsf{u}(f)(\{\omega_k\}_k)=f(\{\omega_{k-1}\}_k)$$ for $f\in
L^p(\Omega_0).$ Clearly, $\mathsf{u}$ is a positive isometric isomorphism. Therefore one can extend $\mathsf{u}$ to an operator $\mathsf{U}:=\mathsf{u}\otimes I_X$ on the Banach space $L^p(\Omega_0,X)$ as an isometric 
isomorphism. Since $\mathsf{u}(\epsilon_k)=\epsilon_{k-1}$, for any $\sum_{k}\epsilon_k\otimes x_k\in\text{Rad}_p(X),$ we have  $\mathsf{U}(\sum_k\epsilon_k\otimes x_k)=\sum_k
\epsilon_k\otimes x_{k+1}.$ Let $U$ be the isometric isomorphism  $I_X\oplus\mathsf{U}$ on $X\oplus_p L^p(\Omega_0,X).$ Note that one can identify $X\oplus_p L^p(\Omega_0,X)$ as $L^p(\Omega^\prime,X)$ for 
some measure space $\Omega^\prime.$ 

Since the Banach spaces $X$ and $X^*$ are of finite cotype and $T_1$ admits a bounded $H^\infty$ functional calculus, we have the square function estimates \cite{AFL}
$\|x\|_{T_1,\frac{1}{2}}\leq C_1\|x\|_{X}$ and $\|y\|_{T_1^*,\frac{1}{2}}
\leq C_2\|y\|_{X^*}$ for all $x\in X,\ y\in X^*.$ Again as $T_1$ is power bounded and $X$ is reflexive, one can use mean ergodic theorem \cite{KR} to have the following
decompositions, $$X={\rm Ker}(I_X-T_1)\oplus\overline{{\rm Ran}(I_X-T_1})$$ and $$X^*={\rm Ker}(I_{X^*}-T_1^*)\oplus\overline{{\rm Ran}(I_{X^*}-T_1^*}).$$ Also one can notice that for any $x_0\in {\rm Ker}(I_X-T_1),x_1\in \overline{{\rm Ran}(I_X-T_1})$ and $y_0\in {\rm Ker}(I_{X^*}-T_1^*),y_1\in \overline{{\rm Ran}(I_{X^*}-T_1^*}),$ we have $\langle x_0,y_1\rangle=\langle x_1,y_0\rangle=0.$ Therefore, from the square function estimates, we can define the 
linear map $$J:{\rm Ker}(I_X-T_1)\oplus\overline{{\rm Ran}(I_X-T_1})\to X\oplus_pL^p(\Omega_0,X),$$ as ${J}(x_0\oplus x_1)=x_0\oplus\sum_{k=0}^\infty\epsilon_k\otimes T_1^k(I_X-T_1)^{\frac{1}{2}}x_1.$

 Similarly, define 
$$\tilde{J}:{\rm Ker}(I_{X^*}-T_1^*)\oplus\overline{{\rm Ran}(I_{X^*}-T_1^*})\to X^*\oplus_{p^\prime}L^{p^\prime}(\Omega_0,X^*),$$ as $\tilde{J_1}(y_0\oplus y_1)=y_0\oplus\sum_{k=0}^\infty\epsilon_k\otimes{T_1^*}^k(I_{X^*}-T_1^*)
^{\frac{1}{2}}y_1.$

 Following the proof of \cite{AFL} Theorem (4.1), we have that \[\langle U^nJ(x_0\oplus x_1),\tilde{J}(y_0\oplus y_1)\rangle=\langle x_0,x_1\rangle+\langle(I_X+T_1)^{-1}T_1^nx_1,y_1\rangle.\] Let us define the operator
$\Theta:{\rm Ker}(I_X-T_1)\oplus\overline{{\rm Ran}(I_X-T_1})\to X\oplus_pL^p(\Omega_0,X)$ as \[\Theta(x_0\oplus x_1):=x_0\oplus(I_X+T_1)x_1.\] Thereafter, it is easy to check that $\langle U^nJ\Theta(x_0\oplus x_1),\tilde{J}
(y_0
\oplus y_1)\rangle=\langle x_0,y_0\rangle+\langle T_1^nx_1,y_1\rangle.$ Define $Q_1={\tilde{J}}^*$ and $J_1=J\Theta$ to obtain $T_1^n=Q_1U^nJ_1,n\geq 0.$ 
%\begin{displaymath}
%\xymatrix{
%X\ar[rr]^{T_1^n}  \ar[d]^{J_1}  &&
%X \\
%X\oplus_p L^p(\Omega_0, X)\ar[rr]^{I_X\oplus (\mathsf{u}^n\otimes I_X)}&& X\oplus_p  L^p(\Omega_0,X)\ar[u]^{Q_1}
%}
%\end{displaymath}

We notice the following identity,
\begin{equation}\label{ID} J_1S=(S\oplus(I_{L^p(\Omega_0)}\otimes S))J_1,\end{equation} where $S:X\to X$ is a bounded operator which commutes with $T_1$. As, for   $x_0\in {\rm Ker}(I_X-T_1),\;{\rm and}\; x_1\in \overline{{\rm Ran}(I_X-T_1})$ we have,
\begin{eqnarray*}
&&(S\oplus(I_{L^p(\Omega_0)}\otimes S))J\Theta(x_0\oplus x_1)\\
&=& (S\oplus(I_{L^p(\Omega_0)}\otimes S))J(x_0\oplus(I_X+T_1)x_1)\\
&=&(S\oplus(I_{L^p(\Omega_0)}\otimes S))(x_0\oplus\sum_{k=0}^\infty\epsilon_k\otimes T_1^k(I_X-T_1)^k)^{\frac{1}{2}}(I_X+T_1)x_1)\\
&=& Sx_0\oplus\sum_{k=0}^\infty\epsilon_k\otimes T_1^k(I_X-T_1)^k)^{\frac{1}{2}}(I_X+T_1)Sx_1\\
&=&J_1S(x_0\oplus x_1).
\end{eqnarray*}
Let $(T_2,\dots,T_{m})$ be commuting tuple of Ritt operators each of which admits a bounded $H^\infty$-functional calculus. Hence by induction hypothesis, there exists  a measure space $\Omega^{\prime\prime},$ a 
commuting tuple of isometric isomorphisms $(U_2,\dots,U_{m})$ on $L^p(\Omega^{\prime\prime},X)$ together with two bounded operators $Q_2:L^p(\Omega^{\prime\prime},X)\to X$ and $J_2:X\to L^p(\Omega^{\prime\prime},X)$  such that 
$$T_2^{i_2}\cdots T_{m}^{i_{m}}=Q_2U_2^{i_1}\cdots U_{m}^{i_{m}}J_2$$ for all $i_2,\dots,i_{m}\in\mathbb{N}_0.$ We notice the following,
\begin{eqnarray*}
&&T_1^{i_1}T_2^{i_2}\dots T_m^{i_m}\\&=& Q_1U^{i_1}J_1T_2^{i_2}\cdots T_{m}^{i_{m}}\\
&=&Q_1U^{i_1}(T_2^{i_2}\cdots T_{m}^{i_{m}}\oplus(I_{L^p(\Omega_0)}\otimes T_2^{i_2}\cdots T_{m}^{i_{m}}))J_1\\
&=&Q_1U^{i_1}(Q_2\oplus(I_{L^p(\Omega_0)}\otimes Q_2))\Big(\prod_{j=2}^m(U_j^{i_j}\oplus(I_{L^p(\Omega_0)}\otimes U_j^{i_j}))\Big)\\
&&\hspace*{2in}(J_2\oplus(I_{L^p(\Omega_0)}\otimes J_2))J_1\\
&=&\mathcal{Q}(I_{L^p(\Omega^{\prime\prime},X)}\oplus(\mathsf{u}^{i_1}\otimes I_{L^p(\Omega^{\prime\prime},X)}))\Big(\prod_{j=2}^m(U_j^{i_j}\oplus(I_{L^p(\Omega_0)}\otimes U_j^{i_j}))\Big)\mathcal{J},
\end{eqnarray*} where $\mathcal{Q}=Q_1(Q_2\oplus(I_{L^p(\Omega_0)}\otimes Q_2))$ and $\mathcal{J}=(J_2\oplus(I_{L^p(\Omega_0)\otimes J_2}))J_1.$ 

If $X$ is an ordered Banach space, a closed subspace of an $L^p$-space, or an $\text{SQ}_p$ space, one can get the result similarly.

\end{proof}
\begin{center}
 {\bf Proofs of  Theorem \eqref{CLASS} and  Theorem \eqref{JRS} }
\end{center}
We will use method of transference for proving the implication $(ii)\Rightarrow (i)$ in Theorem (\ref{CLASS}). 
%We need following facts and results for our purpose. 
%\begin{definition}[Fourier multiplier]\label{FM}
%Let $G$ be a locally compact abelian group with its dual group $\widehat{G}$. For $1\leq p < \infty,$ a bounded operator $T:L^p(G)\to L^p(G)$ is called a Fourier multiplier, if there exists a $\phi\in L^\infty(\widehat{G})$ such that   $\widehat{Tf}=\phi\widehat{f} $ for  $f\in L^2(G)\cap L^p(G).$
%\end{definition}
%We denote such an operator $T$ by $M_{\phi}$. We represent the set of all Fourier multipliers on $G$ by $M_p(G)$. It is well known that the set $M_p(G)$ equipped with the operator norm is a Banach algebra under pointwise multiplication. For more on multiplier theory on locally compact abelian groups, we recommend the reader \cite{LA}. We also need following variant of Lemma 3.3 in \cite{RA1}, which can also be of independent interest and proved by similar argument.

\begin{proof}[\bf Proof of Theorem (\ref{CLASS})]
 Form Theorem \eqref{DIAL}  we have  $\mathbf T$ admits joint isometric loose dilation.  As $T_i$ satisfies $H^\infty$-functional calculus for each $i=1,\dots,n$ by \cite{ALM} we have  $T_i$'s are $R$-Ritt. This proves $(1)\Rightarrow (2)$. 
 
%For the implication of
Assuming $(2),$ one can deduce $(3)$ by using a matrix-valued version of Coifmann-Weiss general transference principle (Theorem 5.2.1, \cite{FE}).
$(3)\Rightarrow(4)$ is trivial. 
Now for the remaining we have each $T_i$ is $R$-Ritt and $p$-polynomially bounded by \cite{ALM}, $I-T_i$ has bounded $H^\infty(\Sigma_{\theta_i})$-functional calculus for some $\theta_i\in(0,\pi).$ Again by Theorem \eqref{trans}, \cite{ALM} and \cite{FM} we have the required result.
\end{proof}
It is known that Ritt operators which are positive contractions admit  bounded $H^\infty$-functional calculus \cite{AFL}, \cite{LX}. So, by applying Theorem \eqref{CLASS} we have the following weak version of Akcoglu-Sucheston dilation theorem. 
\begin{corollary}
Let $\mathbf{T}=(T_1,\dots,T_n)$ be an $n$-tuple of commuting Ritt operators which are also positive contraction on an $L^p$-space, $1<p<\infty.$ Then $\mathbf T$ admits a joint isometric loose dilation. 
\end{corollary}

%\begin{remark}
%Note that in above we use Theorem \eqref{trans} to prove that on $L^p$-spaces, $1<p<\infty$ if we have a commuting tuple of Ritt operators each of which admits a bounded $H^\infty$-functional calculus then the tuple admits a joint bounded $H^\infty$-functional calculus. It is possible to prove this fact  by developing a Franks-Mcintosh \cite{FM} type decomposition of holomorphic functions defined on Stolz domain.
%\end{remark}

\begin{proof}[\bf Proof of Theorem 
(\ref{JRS})]
 If each $T_i$ is similar to contraction then by \cite{LeM2} (Theorem 8.1) each $T_i$
admits a bounded $H^{\infty}$-functional calculus. Thus, $(2)$ follows by \cite{FM}. Assuming $(2)$, by Theorem \eqref{DIAL}, $\mathbf{T}$ has a joint isometric loose dilation on $L^2(\Omega,\mathcal{H})$ for some measure space $\Omega$. Therefore, $(2)\implies (3)$ by \cite{PA1} (Corollary 5.2). $(3)\implies(1)$ is trivial.
\end{proof}
\textbf{Acknowledgement:} We are thankful to  Prof. C. Le Merdy for exposing us to the theory of Ritt operators and suggesting the problem as well as  many valuable discussions and suggestions during this work.  We appreciate   Sayantan Chakrabarty's help in drawing the pictures.

\end{document}